\documentclass[11pt, reqno]{amsart}
\usepackage{amssymb, amsthm, amsmath, amsfonts,mathtools}
\usepackage{array, epsfig}
\usepackage{bbm}
\usepackage{MnSymbol}

\usepackage{hyperref}
\usepackage[numbers,square]{natbib}
\usepackage{setspace}
\usepackage[english]{babel}
\usepackage[utf8]{inputenc}     
\usepackage{thmtools}
\usepackage{url}

\newcommand{\Lah}{\mathop{\mathrm{Lah}}\nolimits}

\renewcommand{\P}{\mathbb {P}} 
\newcommand{\E}{\mathbb {E}}   
\newcommand{\C}{\mathbb {C}}   
\newcommand{\R}{\mathbb {R}}   
\newcommand{\Z}{\mathbb {Z}}   
\newcommand{\N}{\mathbb {N}}   
\newcommand{\ind}{\mathbbm{1}} 

\renewcommand{\L}{\mathbb{L}}

\newcommand{\e}{\mathrm{e}}
\newcommand{\diff}{\mathop{}\!{d}}

\renewcommand{\Re}{\operatorname{Re}}
\renewcommand{\Im}{\operatorname{Im}}

\DeclareMathOperator{\Var}{Var}

\newcommand{\todistr}{\overset{{d}}{\underset{n\to\infty}\longrightarrow}}

\newcommand{\ton}{\overset{}{\underset{n\to\infty}\longrightarrow}}

\theoremstyle{plain}
\newtheorem{theorem}{Theorem}[section]
\newtheorem{satz}[theorem]{Theorem}
\newtheorem{lem}[theorem]{Lemma}
\newtheorem{kor}[theorem]{Corollary}
\newtheorem{prop}[theorem]{Proposition}

\theoremstyle{definition}
\newtheorem{dfn}[theorem]{Definition}

\newtheorem{bem}[theorem]{Remark}

\theoremstyle{remark}

\numberwithin{equation}{section}

\newcommand{\stirsec}[2]{{ #1 \brace #2}}
\newcommand{\stirfir}[2]{{ #1 \brack #2}}

\makeatletter
\@namedef{subjclassname@2020}{\textup{2020} Mathematics Subject Classification}
\makeatother

\begin{document}

\onehalfspacing
\pagenumbering{roman}
\setcounter{page}{1}
\doublespacing
\singlespacing
\pagenumbering{arabic}
\raggedbottom
\onehalfspacing

\title[$r$-Lah distribution]{$r$-Lah distribution: properties, limit theorems and an application to compressed sensing}

\begin{abstract}
We introduce and study the  $r$-Lah distribution whose definition involves $r$-Stirling numbers of both kinds. We compute its expectation and  variance, show its log-concavity and prove limit theorems for this distribution. We use these results to prove threshold phenomena for convex cones generated by random walks and to analyze the probability of unique recovery of sparse monotone signals from linear measurements.
\end{abstract}

\author{Zakhar Kabluchko and David Albert Steigenberger}

\maketitle

\section{Introduction}
The goal of this paper is to introduce and study the $r$-Lah distribution.
The motivation  comes from the following problem (c.f.\ Section~6 in \cite{kabluchko:2022}). Let $\xi_1 ,\dots, \xi_n$ be a collection of i.i.d.\ random $d$-dimensional vectors with absolutely continuous, centrally symmetric distribution (the latter means that $\xi_1$ has the same law as $-\xi_1$). Their partial sums are denoted by
\begin{align*}
S_i \coloneqq \xi_1+\dotsm+\xi_i, \qquad 1 \leq i \leq n,
\qquad
S_0 = 0.
\end{align*}
The sequence $S_0,S_1,\dotsc,S_n$ is called a random walk.
We are interested in  the random polyhedral cone $C_n^B\subseteq \R^d$ defined by
\begin{align*}
 C_n^B \coloneqq \text{pos}\lbrace S_1, \dotsc S_n \rbrace = \lbrace \lambda_1 S_1 + \dotsm + \lambda_n S_n \;| \; \lambda_1 , \dotsc , \lambda_n \geq 0 \rbrace,
\end{align*}
which is referred to as the Weyl random cone of type B+ due to its relation to the Weyl chamber of type $B$, see \cite{godland:2022}.
Corollary~2.12 in \cite{godland:2022} (see also~\cite[Eqn.~(2.23)]{godlandschlaefli:2022}) states that the expected total number of $k$-dimensional faces of $C_n^B$ is
\begin{align}\label{Cones}
\E \Bigl[ f_k \bigl( C_n^B \bigr) \Bigr] = \frac{2 \cdot k!}{n!} \sum_{l=0}^\infty \stirfir{n}{d-2l-1}_{\frac{1}{2}}\stirsec{d-2l-1}{k}_{\frac{1}{2}},
\end{align}
for all $k\in \{0,\ldots, d-1\}$ and $n\geq d$,
where the numbers on the right-hand side are the $r$-Stirling numbers with parameter $r=1/2$, to be defined in Section~\ref{ChapterrStirling}.  Our goal is to investigate the relation between the expected number of $k$-dimensional faces of $C_n^B$ and the maximal possible number $\binom{n}{k}$ of such $k$-faces.
Namely, if $d\to\infty$ and $n=n(d)\to \infty$ grows sufficiently slow, then we expect the number of $k$-faces of $C_n^B$ to be ``close'' to the maximal number $\binom nk$, while in the regime when $d\to\infty$ and $n = n(d)\to \infty$ grows sufficiently fast, we expect it to be much smaller than $\binom nk$. We will confirm this intuition and identify the position of the threshold where the transition occurs. To this end, we introduce  the $r$-Lah distribution which is closely related to the right-hand side of~\eqref{Cones}. This distribution generalizes the Lah distribution, which has been studied in~\cite{kabluchko:2022} and corresponds to $r=0$.  As it turns out, limit theorems for the $r$-Lah distribution lead to threshold phenomena for the face numbers of $C_n^B$. Threshold phenomena of this type were first investigated by Vershik and Sporyshev~\cite{vershik_sporyshev_asymptotic_faces_random_polyhedra1992} and Donoho and Tanner~\cite{donoho_neighborliness_proportional,donoho_tanner_neighborliness,donoho_tanner,donoho:2010} in the setting of random projections of regular simplices, cubes, orthants, and Gaussian polytopes.

This paper is organized as follows.
In Section \ref{ChapterrStirling} we will recall the definition of the $r$-Stirling numbers and use these to define the $r$-Lah distribution. In Section \ref{ChapterProperties} we derive explicit formulas for the expectation and the variance of this distribution and prove its log-concavity. Furthermore, we prove several limit theorems including a central limit theorem. These limit theorems will then be used to prove threshold phenomena for the expected number of $k$-faces of $C_n^B$ in Section~\ref{threshold_phenomena}. As another application, in Section~\ref{sec:compressed_sensing} we analyze threshold phenomena for the probability of unique recovery of sparse monotone signals from  linear measurements.

\section{\texorpdfstring{$r$}{r}-Stirling numbers and the \texorpdfstring{$r$}{r}-Lah distribution}\label{ChapterrStirling}
The Stirling number of the first kind, denoted by $\stirfir{n}{k}$, counts the number of permutations of $\lbrace 1, \dotsc ,n \rbrace$ that have exactly $k$ cycles. The Stirling number of the second kind counts the number of ways to partition the set $\lbrace 1, \dotsc ,n \rbrace$ into $k$ non-empty subsets and is denoted by $\stirsec{n}{k}$. For $k \in \N_0$, these numbers can be  defined~\cite[Eqs.~(2.23), (2.19)]{mezo_book} by their exponential generating functions
\begin{align}\label{Definition_exp_generating_Stir}
\frac{1}{k!} \left( \log\left( (1-x)^{-1} \right) \right)^k = \sum_{n=k}^\infty \stirfir{n}{k} \frac{x^n}{n!}
\quad \text{ and } \quad \frac{1}{k!} \left(\e^x-1\right)^k = \sum_{n=k}^\infty \stirsec{n}{k} \frac{x^n}{n!}.
\end{align}
We use the usual Karamata notation for Stirling numbers~\cite{knuth_notation}. We refer to~\cite{graham_etal_book,mezo_book,charalambides_book_enum_combin,charalambides_book_combinatorial_methods} for information on Stirling numbers and their probabilistic applications.
The Lah number~\cite{lah:1954,petkovsek:2022} is defined by
\begin{align*}
L(n,k) \coloneqq \sum_{j=k}^n \stirfir{n}{j} \stirsec{j}{k}
\end{align*}
and counts the number of ways to partition the set $\lbrace 1, \dotsc ,n \rbrace$ into $k$ non-empty subsets and to linearly order the elements inside each subset.

One can generalize~\cite{broder:1984,merris:2000,mezo_book,nyul:2015} all three kinds of numbers in the following way: for $r \in \N_0$, the $r$-Stirling number of the first kind $\stirfir{n}{k}_r$ counts the number of permutations of the set $\lbrace 1,\dotsc ,n+r \rbrace$ with exactly $k+r$ cycles, such that the elements $1,\dotsc ,r$ are in distinct cycles. The $r$-Stirling number of the second kind, $\stirsec{n}{k}_r$, as well as the $r$-Lah number, $L(n,k)_r$, can be defined analogously\footnote{We follow the notation of  Nyul and R\'acz~\cite{nyul:2015}:  what we denote by $\stirfir{n}{k}_r$ corresponds to $\stirfir{n+r}{k+r}_r$ in the notation of Broder~\cite{broder:1984}, and similarly for $\stirsec{n}{k}_r$.}.
The exponential generating functions of the $r$-Stirling numbers are given by
\begin{align}\label{Definition_r_Stir_first}
\frac{1}{k!} \left( \log\left( (1-x)^{-1} \right) \right)^k (1-x)^{-r} = \sum_{n=k}^\infty \stirfir{n}{k}_r \frac{x^n}{n!} \quad
\end{align}
and
\begin{align}\label{Definition_r_Stir_second}
\quad \frac{1}{k!} \left(\e^x-1\right)^k \e^{rx} = \sum_{n=k}^\infty \stirsec{n}{k}_r \frac{x^n}{n!}.
\end{align}
It is clear from these formulas that the regular Stirling numbers are recovered by taking $r=0$.
Using~\eqref{Definition_exp_generating_Stir} and the Taylor series of $\e^{rx}$ and $(1-x)^{-r}$ one gets (c.f.~\cite[p.~1661]{nyul:2015})
\begin{align}\label{Definition_r_STirl_as_polynomial}
\stirfir{n}{k}_r
=
\sum_{j=k}^n \stirfir{n}{j} \binom{j}{k} r^{j-k}
\qquad \text{ and } \qquad
\stirsec{n}{k}_r
=
\sum_{j=k}^n \binom{n}{j} \stirsec{j}{k} r^{n-j}
.
\end{align}
These formulas define the $r$-Stirling numbers for $n\in \N_0$, $k\in \{0,\ldots, n\}$ and arbitrary non-integer $r\in \R$. Furthermore, they prove that $\stirfir{n}{k}_r$ and $\stirsec{n}{k}_r$  are polynomials in $r$.
The $r$-Stirling numbers satisfy the recurrence relations
\begin{align}
\stirfir{n}{k}_r  &= (n+r-1)\stirfir{n-1}{k}_r + \stirfir{n-1}{k-1}_r, \label{recurrence_stirling1}\\
 \stirsec{n}{k}_r &= (k+r)\stirsec{n-1}{k}_r + \stirsec{n-1}{k-1}_r. \label{recurrence_stirling2}
\end{align}
A combinatorial proof for $r\in \N_0$ can be found in~\cite[Theorems~1,2]{broder:1984}, but the relations stay true for all $r\in \R$ since both sides are polynomials in $r$.
Notable special cases are $\stirfir{n}{0}_r= r(r+1)\dotsm (r+n-1)$ and $\stirsec{n}{0}_r = r^n$. In particular, $\stirfir{0}{0}_r = \stirsec{0}{0}_r = 1$  for all $r\in \R$. For integer $k>n$ and $k<0$ we put $\stirfir{n}{k}_r=\stirsec{n}{k}_r = 0$, for all $r\in \R$, $n\in \N_0$. For these and further  properties of the $r$-Stirling numbers we refer to~\cite{broder:1984,nyul:2015} and~\cite[Chapter~8]{mezo_book}.

The $r$-Lah number~\cite{nyul:2015}, see also \cite{belbachir:2013,cheon:2012,shattuck:2016}, is defined as
\begin{align}\label{eq:Lah_r_def}
L(n,k)_r
\coloneqq \sum_{j=k}^n \stirfir{n}{j}_r \stirsec{j}{k}_r
=
\binom{n+2r-1}{k+2r-1} \frac{n!}{k!},
\end{align}
for $n\in \N_0$, $k\in \{0,\ldots, n\}$, $r\in \R$.
The second equation has been proven for $r\in \N_0$ in Theorem 3.7 in \cite{nyul:2015}
but it stays true for $r\in \R$ since both sides are polynomials in $r$.  We are now in a position to  define the $r$-Lah distribution.

\begin{dfn}\label{DfnrLah}
A random variable $X=\Lah(n,k)_r$ has an $r$-Lah distribution with parameters $n \in \N$, $k\in \{0,\ldots, n\}$ and $r \in [0,\infty)$ (where the case $k=r=0$ is always excluded) if
\begin{align}\label{eq:Lah_distr_def}
\P [\Lah(n,k)_r=j] = \frac{1}{L(n,k)_r} \stirfir{n}{j}_r \stirsec{j}{k}_r, \qquad  j \in \lbrace k , k+1, \dotsc, n \rbrace.
\end{align}
\end{dfn}

We exclude the case $k=r=0$ in which $L(n,0)_0 = 0$, $n\in \N$,  since  $(-1)! = \infty$. For $r=0$, the $r$-Lah distribution coincides with the Lah distribution that has been studied in \cite{kabluchko:2022}.  The present paper extends these results to arbitrary real $r\geq 0$.

\begin{dfn}\label{DfnrParamLah}
We call  $(n,k,r)$ an  admissible triple of parameters for the $r$-Lah distribution if  $n \in \N$, $k\in \{0,\ldots, n\}$, $r \in [0,\infty)$ and $\max\{k,r\} > 0$.
\end{dfn}

\begin{bem}
The term $\stirfir{n}{j}_r \stirsec{j}{k}_r$ counts the number of permutations of $\{1,2,\ldots, n+r\}$ with $j+r$ cycles which are decomposed into $k+r$ blocks of cycles and have the property that the elements $1,\ldots, r$ appear in distinct blocks of cycles (and, consequently, in distinct cycles). The well-known Foata correspondence is a bijection between permutations of $\{1,\ldots, m\}$ with $i$ cycles and permutations of the same set with $i$ minima. Applying the Foata correspondence to each block of cycles yields a decomposition of $\{1,\ldots, n+r\}$ into $k+r$ blocks, with a specified ordering of the elements inside each block, with $j+r$ local minima inside the blocks  and such that the elements $1,\ldots, r$ are inside different blocks. It follows from~\eqref{eq:Lah_r_def} and~\eqref{eq:Lah_distr_def} that $\Lah(n,k)_r-r$ is the distribution of the number of local minima inside a uniformly random decomposition as just described. This interpretation stems from the combinatorial explanation of \eqref{eq:Lah_r_def} established in  Theorem~4.1 (iv) in \cite{shattuck:2016} for $r=s$.
\end{bem}

\section{Properties of the \texorpdfstring{$r$}{r}-Lah distribution}\label{ChapterProperties}
In this section, we will  derive explicit and asymptotic formulas for the expectation and the variance of the $r$-Lah distribution.  We will show that the $r$-Lah distribution is log-concave and thus unimodal.  Finally, we will use the concept of mod-Poisson convergence to prove limit theorems for the $r$-Lah distribution, namely central and local limit theorems, precise large deviations and an asymptotic formula for its mode.

\subsection{Expectation and variance}

\begin{satz}\label{expectationrLah}
For any admissible triple of parameters $(n,k,r)$ the expectation of an $r$-Lah distributed random variable $\Lah(n,k)_r$ is given by
\begin{alignat}{3}
&\E [\Lah(n,k)_r] &=& \frac{k(n+2r)}{n-(k-1)} \bigl[ H_{n+2r} - H_{k+2r-1} \bigr] + r \bigl[ H_{n+2r-1} - H_{k+2r-1} \bigr] \label{ExpectrLahfirst} \\
&       &=& \frac{k+ \bigl[ k(n+r) + r(n+1) \bigr] \cdot \bigl[ H_{n+2r-1} - H_{k+2r-1} \bigr]}{n-(k-1)}. \label{ExpectrLahsecond}
\end{alignat}
\end{satz}

Here, $H_n= 1+ \frac 12 +\dotso +\frac 1n$ denotes the $n$-th harmonic number. The harmonic number of fractional argument $z$ can be defined by $H_{z} = \Gamma'(z+1)/\Gamma(z+1) + \gamma$, but, actually, this is not necessary since the above formulas contain only differences of the form  $H_{\alpha + m} - H_{\alpha} :=  \sum_{j=\alpha+1}^{\alpha+m} \frac{1}{j}$ with integer $m\in \N_0$ and $\alpha>-1$.

For  $r=0$ Theorem~\ref{expectationrLah} yields the following corollary which has been overlooked in~\cite[Theorem 3.2]{kabluchko:2022}, where a more lengthy expression has been given.

\begin{kor}
The expectation of a Lah distributed random variable with parameters $n \in \N$, $k\in \{1,\ldots, n\}$ (and $r=0$) is given by
\begin{align*}
\E[\Lah(n,k)_0] = \frac{nk}{n-(k-1)} [H_{n} - H_{k-1}].
\end{align*}
\end{kor}

In order to prove Theorem \ref{expectationrLah} we need two lemmata.

\begin{lem}\label{rstirexpgenproduct}
For all $k\in \N_0$  and $r\in \R$ we have
\begin{align*}
\sum_{j=k}^{\infty} \sum_{n=j}^{\infty} \frac{1}{n!} \stirfir{n}{j}_r \stirsec{j}{k}_r t^j x^n = \frac{1}{k!} \bigl((1-x)^{-t}-1\bigr)^k (1-x)^{-(rt+r)}.
\end{align*}
\end{lem}

\begin{proof}
We compute the Taylor series of the function on the right-hand side using the exponential generating functions of the $r$-Stirling numbers:
\begin{align}
\frac{1}{k!} \bigl((1-x)^{-t}-1\bigr)^k (1-x)^{-(rt+r)}
&=
(1-x)^{-r} \Biggl[ \frac{1}{k!} \Bigr( \e^{- t \log (1-x) }-1 \Bigl)^k \e^{-rt \log (1-x) } \Biggr]
\notag \\
&\stackrel{\eqref{Definition_r_Stir_second}}{=}
(1-x)^{-r} \sum_{j=k}^{\infty} \Biggl[\stirsec{j}{k}_r  \frac{\bigl(- t \log (1-x) \bigr)^j}{j!} \Biggr]
\notag \\
&=
\sum_{j=k}^{\infty} \stirsec{j}{k}_r t^j \Biggl[ \frac{ \left(-\log (1-x)\right)^j}{j!} (1-x)^{-r} \Biggr]
\notag \\
&\stackrel{\eqref{Definition_r_Stir_first}}{=}
\sum_{j=k}^{\infty} \stirsec{j}{k}_r t^j \Biggl[ \sum_{n=j}^{\infty} \stirfir{n}{j}_r \frac{x^n}{n!} \Biggr]
\notag \\
&=
\sum_{j=k}^{\infty} \sum_{n=j}^{\infty} \frac{1}{n!} \stirfir{n}{j}_r \stirsec{j}{k}_r t^j x^n.
\notag
\end{align}
This gives the statement.
\end{proof}

Next we will use Lemma \ref{rstirexpgenproduct} to derive the formula that will help us calculate the expectation of the $r$-Lah distribution. As usual, $[x^n]f(x)$ denotes the coefficient of $x^n$ in the Taylor/Laurent  expansion of the function $f$.

\begin{lem}
For all $n\in \N$, $k\in \{0,\ldots, n\}$ and $r\in \R$ we have
\begin{alignat}{2}
\sum_{j=k}^n j \stirfir{n}{j}_r \stirsec{j}{k}_r &= (-1)^{n-k}\frac{n!}{k!} [x^{n-k+1}] \bigl( (1+x)^{-(k+2r)} \log(1+x)(k-rx) \bigr) \label{ExpectationrLahfirst} \\
&=  \frac{n!}{k!} [x^{n-k+1}] \Biggl( (1+x)^{n+2r} \log(1+x) \biggl( k+r \frac{x}{1+x} \biggr) \Biggr). \label{ExpectationrLahsecond}
\end{alignat}
\end{lem}

\begin{proof}
Our starting point is the formula of Lemma~\ref{rstirexpgenproduct}.
The function $( (1-x)^{-t} -1)^{k} (1-x)^{-(rt+r)}$ is analytic in $(x,t)$ as a composition of analytic functions as long as $|x|<1$, $t\in \C$.   Therefore, we may differentiate its Taylor series any number of times in $x$ and $t$ and interchange the order of derivatives.
We differentiate in $t$ and set $t=1$:
$$
\sum_{n=k}^\infty \sum_{j=k}^n j \frac{k!}{n!}  \stirfir{n}{j}_r \stirsec{j}{k}_r x^n  = - \biggl(\frac{x}{1-x} \biggr)^{k-1} (1-x)^{-2r-1} (\log(1-x)) (k+rx).
$$
Dividing by $x^{k-1}$ and evaluating the coefficient of $x^{n-k+1}$, we get
$$
\sum_{j=k}^n j \stirfir{n}{j}_r \stirsec{j}{k}_r  = -\frac{n!}{k!} [x^{n-k+1}] (1-x)^{-(k+2r)} (\log(1-x))(k+rx).
$$
Changing $x$ to $(-x)$ on the right-hand side yields Equation~\eqref{ExpectationrLahfirst}
\begin{align*}
\sum_{j=k}^n j \stirfir{n}{j}_r \stirsec{j}{k}_r = (-1)^{n-k} \frac{n!}{k!} [x^{n-k+1}]  \bigl( 1+x \bigl)^{-(k+2r)} \log(1+x)(k-rx).
\end{align*}
To prove Equation~\eqref{ExpectationrLahsecond} we use Cauchy's integral formula giving
\begin{align*}
\sum_{j=k}^n j \stirfir{n}{j}_r \stirsec{j}{k}_r = (-1)^{n-k} \frac{n!}{k!} \frac{1}{2\pi i} \oint_{\gamma} \frac{\log(1+x)}{(1+x)^{k+2r}} \frac{(k-rx)}{x^{n-k+2}} \diff x.
\end{align*}
Here, $\gamma$ is a small counterclockwise circle centered at zero. Now, we substitute $1+x$ with $\frac{1}{1+y}$ to get
\begin{align*}
\sum_{j=k}^n j \stirfir{n}{j}_r \stirsec{j}{k}_r =  \frac{n!}{k!} \frac{1}{2\pi i} \oint_{\gamma '}  \frac{\log(1+y)(1+y)^{n+2r}}{y^{n-k+2}} \Biggr(k+r\frac{y}{1+y} \Biggr) \diff y.
\end{align*}
Using Cauchy's integral formula once more, we arrive at Equation (\ref{ExpectationrLahsecond}).
\end{proof}
\begin{proof}[Proof of Theorem~\ref{expectationrLah}]
We start with (\ref{ExpectationrLahsecond}) and realize that
\begin{multline*}
\frac{k!}{n!}\sum_{j=k}^n j \stirfir{n}{j}_r \stirsec{j}{k}_r
= k\bigl[x^{n-k+1}\bigr]\left((1+x)^{n+2r} \log(1+x)\right)
\\
 + r\bigl[x^{n-k}\bigr]\left((1+x)^{n+2r-1} \log(1+x)\right).
\end{multline*}
Using the Taylor series $\log(1+x) = \sum_{m=1}^\infty (-1)^{m-1} \frac{x^{m}}{m}$ and $(1+x)^\nu = \sum_{p=0}^\infty  \binom{\nu}{p} x^p$ which converge for $|x| < 1$
we write
\begin{alignat*}{3}
\log(1+x) (1+x)^\nu  =\sum_{z=1}^\infty & x^{z} \sum_{m=1}^{z} (-1)^{m-1} \frac{1}{m} \binom{\nu}{z-m}.
\end{alignat*}
Then,
\begin{alignat*}{3}
\frac{k!}{n!}\sum_{j=k}^n j \stirfir{n}{j}_r \stirsec{j}{k}_r
 =k \sum_{m=1}^{n-k+1} (-1)^{m-1} \frac{1}{m} \binom{n+2r}{n-k+1-m} + r \sum_{m=1}^{n-k} (-1)^{m-1} \frac{1}{m} \binom{n+2r-1}{n-k-m} \\
 =k \frac{n+2r}{n-(k-1)} \binom{n+2r-1}{k+2r-1} \sum_{m=k+2r}^{n+2r} \frac{1}{m} + r \binom{n+2r-1}{k+2r-1} \sum_{m=k+2r}^{n+2r-1} \frac{1}{m},
\end{alignat*}
where we used the identity
\begin{align*}
\sum\limits_{m=1}^z\binom \nu{z-m}\frac{(-1)^{m-1}}m=\binom \nu z\sum\limits_{m=\nu-z+1}^\nu\frac 1 m,
\end{align*}
which holds for all $z \in \N$ and real $\nu>z-1$. The latter condition is secured by $\max\{k,r\}>0$. For integer $\nu$, the identity can be shown by induction. For real $\nu$, it holds since both sides are polynomials in $\nu$.
Multiplying by $1/L(n,k)_r$ and applying \eqref{eq:Lah_r_def} gives us Theorem~\ref{expectationrLah}.
\end{proof}

\begin{satz}\label{AsymptoticsrLah}
Let $n \to \infty$ and $k = k(n) \in \lbrace 1 , \dotsc , n \rbrace$ be a function of $n$. \\
Let $r \geq 0$ be fixed. Then,
\begin{align*}
\E [\Lah(n,k)_r] \sim
\begin{cases}(k+r) \log\left( n/k \right), &\text{if }\quad k=o(n)\\
\frac{n \alpha \log\left( 1/\alpha \right)}{1- \alpha}, &\text{if } \quad k \sim \alpha n \text{ with } \alpha \in (0,1)\\
n, &\text{if } \quad k \sim n,
\end{cases}
\end{align*}
where we write $a_n \sim b_n$ if $a_n / b_n \to 1$ as $n \to \infty$.
\end{satz}

Setting  $r=0$ we recover the formula obtained in Theorem 3.7 in \cite{kabluchko:2022}.

\begin{proof} First, let $k=o(n)$. Using the formula $H_{m+2r-1} = \log  m+ O(1)$ we get $H_{n+2r-1} - H_{k+2r-1} = \log (n/k) + O(1) \sim \log (n/k)$ and the result follows from~\eqref{ExpectrLahsecond} upon observing that $k(n+r) + r(n+1)\sim (k+r) n$.

Let now $k \sim \alpha n$ with $\alpha \in (0,1)$. Then, using the formula $H_{m+2r-1} = \log m + \gamma +o(1)$ as $m\to\infty$, we get $H_{n+2r-1} - H_{k+2r-1} \to \log (1/\alpha)$ and the claim follows from~\eqref{ExpectrLahsecond} after observing that $k(n+r) + r(n+1)\sim \alpha n^2$ and $n-(k-1)\sim (1-\alpha)n$.

For $k \sim n$, we use~\eqref{ExpectrLahfirst}. We notice that
\begin{align*}
H_{n+2r} - H_{k+2r-1} = \frac{1}{k+2r} + \dotso + \frac{1}{n+2r} \in \left[\frac{n-(k-1)}{n+2r},\frac{n-(k-1)}{k+2r} \right],
\end{align*}
which, by the sandwich lemma,  leads to
$$
\lim_{n\to\infty} [H_{n+2r} - H_{k+2r-1}]\frac{(n+2r)}{n-(k-1)}  = 1.
$$
Together with $H_{n+2r-1} - H_{k+2r-1}\to 0$ this yields the claim.
\end{proof}

Another way to derive the formula for the expectation is to adapt the proof of Corollary 2.2 in \citep{van_der_hofstad_etal_shortest_path_trees}. This paper  deals with a distribution which is similar to (but does not coincide with) the $0$-Lah distribution. This approach also allows us to derive a formula for the variance of the $r$-Lah distribution. To this end, we need to simplify the generating function of an $r$-Lah distributed random variable which is defined for complex $t$ by
\begin{align*}
P_{n,k,r}(t) \coloneqq \E \Bigl[t^{\Lah(n,k)_r}\Bigr] = \frac{1}{L(n,k)_r} \displaystyle \sum_{j=k}^n t^j \stirfir{n}{j}_r \stirsec{j}{k}_r.
\end{align*}

\begin{lem}\label{DarstellungPnkr}
For any admissible triple of parameters $(n,k,r)$ and all $t \in \C$ we have
\begin{alignat}{2}
P_{n,k,r}(t) &= \frac{1}{\binom{n+2r-1}{k+2r-1}} [x^n]\bigl((1-x)^{-t}-1\bigr)^k (1-x)^{-(rt+r)} \label{DarstellungPnkrfirst}\\
& =\frac{1}{\binom{n+2r-1}{k+2r-1}} \sum_{m=0}^k (-1)^{k-m} \binom{k}{m} \frac{\Gamma \Bigl( \bigl( r(t+1)+tm+n \bigr) \Bigr)}{\Gamma \Bigl( \bigl( r(t+1)+tm \bigr) \Bigr)n!}. \label{DarstellungPnkrsecond}
\end{alignat}
For $r=0$, the summand with $m=0$ should be interpreted as $0$.
\end{lem}

\begin{proof}
Recall from Lemma~\ref{rstirexpgenproduct} that
\begin{align*}
\sum_{j=k}^n  t^j \frac{k!}{n!} \stirfir{n}{j}_r \stirsec{j}{k}_r = [x^n]\bigl((1-x)^{-t}-1\bigr)^k (1-x)^{-(rt+r)}.
\end{align*}
Dividing this by $\binom{n+2r-1}{k+2r-1}$  gives us~\eqref{DarstellungPnkrfirst}.
To prove~\eqref{DarstellungPnkrsecond}, we apply the binomial formula to~\eqref{DarstellungPnkrfirst}:
\begin{alignat*}{2}
P_{n,k,r}(t)
= \frac{1}{\binom{n+2r-1}{k+2r-1}} [x^n] \sum_{m=0}^k \binom{k}{m}(-1)^{k-m} (1-x)^{-tm} (1-x)^{-(rt+r)} \\
= \frac{1}{\binom{n+2r-1}{k+2r-1}} \sum_{m=0}^k \binom{k}{m}(-1)^{k-m} [x^n] (1-x)^{-( r(t+1)+tm)}.
\end{alignat*}
To complete the proof of~\eqref{DarstellungPnkrsecond}, we use the Taylor expansion of the last term around zero
\begin{alignat*}{2}
&(1-x)^{-( r(t+1)+tm)} = 1 + \sum_{n=1}^{\infty} \frac{\Gamma \Bigl( \bigl( r(t+1)+tm+n \bigr) \Bigr)}{\Gamma \Bigl( \bigl( r(t+1)+tm \bigr) \Bigr)n!} x^n.
\end{alignat*}
\end{proof}
\begin{satz}\label{VariancerLah}
For any admissible triple of parameters $(n,k,r)$ the variance of an $r$-Lah distributed random variable $\Lah(n,k)_r$ is given by
\begin{alignat*}{2}
&\Var[\Lah(n,k)_r] = \frac{1}{(n-k+1)(n-k+2)} \\
&\biggl(  (r(n^2+3n+2)+k^2(n+r)+k(n^2+r+2nr)) (H_{n+2r}-H_{k+2r-1}) \\
&- (r^2(n^2+3n+2)+k^2(n+r)^2+k(2n^2r+r(2+r)+n(1+2r+2r^2)))(H_{n+2r}^{(2)}-H_{k+2r-1}^{(2)}) \\
&-\left(\frac{k(n+1)(n+2r)(k+2r-1)}{n-k+1}\right) (H_{n+2r}-H_{k+2r-1})^2 \biggr) -\frac{rn+r^2}{(n+2r)^2},
\end{alignat*}
where $H_{\alpha +m}^{(2)}-H_{\alpha}^{(2)} = \sum_{i=\alpha +1}^{\alpha +m} \frac 1 {i^2}$, with integer $m\in \N_0$ and $\alpha>-1$.
\end{satz}
\begin{proof}
We can adapt the proof of Corollary 2.2 of \cite{van_der_hofstad_etal_shortest_path_trees}. For $t \in \C$, Lemma~\ref{DarstellungPnkr} gives
\begin{align*}
P_{n,k,r}(t)  &= \frac{1}{\binom{n+2r-1}{k+2r-1}} \sum_{m=0}^k (-1)^{k-m} \binom{k}{m} \frac{\Gamma \Bigl( \bigl( r(t+1)+tm+n \bigr) \Bigr)}{\Gamma \Bigl( \bigl( r(t+1)+tm \bigr) \Bigr)n!} \\
&= \frac{1}{\binom{n+2r-1}{k+2r-1}n!}(-1)^k \partial_x^n[x^{n-1+r}x^{rt}(1-x^t)^k]_{x=1}.
\end{align*}
Taking the second derivative in $t$ at $t=1$ gives
\begin{align*}
\E[\Lah(n,k)^2_r-\Lah(n,k)_r ]
&=
\frac{1}{\binom{n+2r-1}{k+2r-1}n!}(-1)^k \partial_t^2 \partial_x^n[x^{n-1+r}x^{rt}(1-x^t)^k]_{x=t=1} \\
&=
\frac{1}{\binom{n+2r-1}{k+2r-1}n!} (-1)^k \partial_x^n [x^{n-1+r}h(x)],
\end{align*}
where $h(x)=(1 - x)^{k-2} x^r (r^2 (1-x)^2 - 2 k r (1-x) x + k x (k x-1))(\log x)^2$.
As $h$ is a sum of four terms of the form $(1-x)^j x^\beta (\log x)^2$, we can use the general Leibniz rule and the facts that
$\partial_x^i (1-x)^j |_{{x=1}}=j!(-1)^j\delta_{i,j}
$ and
$$
\partial_x^i[x^\beta(\log x)^2]_{x=1}=\frac{\beta!}{(\beta-i)!} \left( \left( \sum_{m=\beta-i+1}^\beta \frac{1}{i}\right)^2 - \sum_{m=\beta-i+1}^\beta \frac{1}{i^2}\right)
$$
(see~\cite[p.~913]{van_der_hofstad_etal_shortest_path_trees} for the latter formula) to arrive at the claim.
\end{proof}
For $r=0$, we get, after rearranging the terms,  a more concise formula.
\begin{kor}\label{VarianceLah}
The variance of a Lah distributed random variable with parameters $n \in \N$, $k\in \{1,\ldots, n\}$ (and $r=0$) is given by
\begin{alignat*}{2}
\Var[\Lah(n,k)_0]&= \frac{n+k}{n-k+2} \E [\Lah(n,k)_0] -\frac{(n+1)(k-1)}{nk(n-k+2)}\E[\Lah(n,k)_0]^2 \\
&-\frac{nk(nk+1)}{(n-k+1)(n-k+2)}(H_{n}^{(2)}-H_{k-1}^{(2)}).
\end{alignat*}
\end{kor}

The next theorem follows from the explicit formula for the variance derived in Theorem \ref{VariancerLah}  analogously to the proof of Theorem~\ref{AsymptoticsrLah}.
Its second  part generalizes Theorem 5.1 of \cite{kabluchko:2022}.
\begin{satz}
Let $n \to \infty$ and $k = k(n) \in \lbrace 1 , \dotsc , n \rbrace$ be a function of $n$. \\
Let $r \geq 0$ be fixed. Then,
\begin{align*}
\Var [\Lah(n,k)_r] \sim
\begin{cases}(k+r) \log\left( n/k \right), &\text{if } k=o(n),\\
-\left(\frac{\alpha}{1-\alpha}+\frac{\alpha(\alpha+1)\log \alpha}{(1-\alpha)^2}+ \frac{\alpha^2(\log\alpha)^2}{(1-\alpha)^3} \right) n, &\text{if }  k \sim \alpha n,\alpha \in (0,1).
\end{cases}
\end{align*}
\end{satz}


\subsection{Mod-Poisson convergence and central limit theorem}
The notion of mod-Poisson convergence has been introduced in \cite{kowalski:2010}, see also~\cite{barbour:2014,delbaen:2015,feray_modphi,jacod:2011,kowalski:2009,meliot:2014}, and provides a unified approach to various limit theorems such as local limit theorems and precise large deviations.

\begin{dfn}\label{DefinitionmodPoisson}
A sequence of random variables $(Z_n)_{n \in \N}$ converges in the  mod-Poisson sense with parameters $(\lambda_n)_{n\in \N}$, where $\lambda_n\to + \infty$, if the limit
\begin{align*}
\Psi(u) := \lim_{n \to \infty} \frac{\E \bigl[ \e^{u Z_n} \bigr]}{\e^{\lambda_n (\e^{u}-1)}}
\end{align*}
exists for all $u$ in some open set $\mathcal{D}\subseteq \C$ containing the real axis $\R$ and the convergence is uniform on compact subsets of $\mathcal D$.
\end{dfn}


The next theorem states that, for fixed $k,r$ and as $n\to\infty$,  the $r$-Lah distribution $\Lah(n,k)_r$ converges in the mod-Poisson sense  with parameters $\lambda_n = (k+r)\log n$.
\begin{satz}\label{rLahismodPoiss}
Let $k \in \N_0$ and $r \in [0, \infty)$ be fixed and such that $\max\{k,r\} > 0$. Then,
\begin{align*}
\lim_{n \to \infty} \frac{\E \Bigl[ \e^{z \Lah(n,k)_r} \Bigr]}{\e^{(k+r)(\log n)(\e^z-1)}} = \frac{\Gamma(k+2r)}{\Gamma((k+r)\e^z+r)}
\end{align*}
for every $z \in \mathcal{D}_{\Lah} \coloneqq \lbrace t \in \C | \cos(\Im t) > 0 \rbrace \supseteq \R$. This convergence is uniform as long as $z$ stays in any compact subset $K$ of $\mathcal{D}_{\Lah}$ and the speed of convergence is $O(n^{- \varepsilon (K)})$ for some $\varepsilon (K) > 0$.
\end{satz}

\begin{proof}
Note that $\E [ \e^{z \Lah(n,k)_r} ] = P_{n,k,r}(\e^z)$.
Since the number of summands in~\eqref{DarstellungPnkrsecond} is fixed,
we investigate each term individually. We shall use that
\begin{align}\label{Gammaalphabeta}
\frac{\Gamma(\alpha + n)}{\Gamma(\beta + n)} = n^{\alpha - \beta}\bigl( 1+O(1/n) \bigr), \qquad \text{ as } n \to \infty,
\end{align}
where the remainder term is uniform as long as $\alpha$ and $\beta$ stay bounded; see~\cite{fields:1970}.
For $t \coloneqq \e^z$ and $m \in \lbrace 0,\dotsc,k \rbrace
$,  the $m$-th summand in~\eqref{DarstellungPnkrsecond} is
\begin{alignat*}{2}
(-1)^{k-m} \binom{k}{m} \frac{\Gamma \Bigl(  r(\e^z+1)+\e^zm+n \Bigr)}{\Gamma \Bigl(r(\e^z+1)+\e^zm \Bigr)n!}
= (-1)^{k-m} \binom{k}{m} \frac{n^{r(\e^z+1)+\e^zm -1} \bigl(1+O(1/n)\bigr)}{\Gamma \Bigl( r(\e^z+1)+\e^zm\Bigr)}
\end{alignat*}
by~\eqref{Gammaalphabeta} with $\alpha = r(\e^z+1)+\e^zm$ and $\beta = 1$.
If $z$ stays in a compact subset $K$ of $\mathcal{D}_{\Lah} \coloneqq \lbrace t \in \C | \cos(\Im t) > 0 \rbrace$, we can find $\varepsilon(K) \in (0,1)$ such that $\Re \e^z > \varepsilon(K) > 0$. Thus, the summand with $m=k$ dominates in the following sense:
\begin{align*}
P_{n,k,r}(\e^z)
&=
\frac{1}{\binom{n+2r-1}{k+2r-1}} \left( \frac{n^{r(\e^z+1)+\e^zk-1}}{\Gamma(r(\e^z+1)+\e^zk)}\bigl(1+O(1/n)\bigr) \right.
\\
&\qquad + \left. \sum_{m=0}^{k-1} (-1)^{k-m} \binom{k}{m} \frac{n^{r(\e^z+1)+\e^zm-1}}{\Gamma(r(\e^z+1)+\e^zm)} \bigl(1+O(1/n)\bigr) \right)\\
&=
\frac{\Gamma(k+2r)}{\Gamma\bigl((k+r)\e^z+r\bigr)}n^{(r+k)(\e^z-1)} \Bigl( 1+O\bigl(n^{-\varepsilon}\bigr) \Bigr).
\end{align*}
Since the choice of $\varepsilon$ only relies on $K$, the error term is uniform on $K$.
\end{proof}
The following result shows that the mod-Poisson convergence implies a central limit theorem (c.f.\ Proposition 2.4(2) in~\cite{kowalski:2010} and its proof).

\begin{satz}\label{modPoissgivesnormal}
Let $(Z_n)_{n \in \N}$ be a sequence of random variables converging in the mod-Poisson sense with parameters $\lambda_n$ such that
$
\lim_{n \to \infty} \lambda_n = \infty.
$
Then,
\begin{align*}
\frac{Z_n - \lambda_n}{\sqrt{\lambda_n}} \todistr \mathcal{N}(0,1).
\end{align*}
\end{satz}
Theorems \ref{rLahismodPoiss} and \ref{modPoissgivesnormal} imply the following Central Limit Theorem for the $r$-Lah distribution.
\begin{satz}[Central Limit Theorem]
\label{thm:clt_lah_r}
Let $k \in \N_0$ and $r \in [0, \infty)$ be fixed and such that $\max\{k,r\} > 0$. Then,
\begin{align*}
\frac{\Lah(n,k)_r - (k+r)\log n}{\sqrt{(k+r)\log n}} \todistr \mathcal{N}(0,1).
\end{align*}
\end{satz}
One can obtain an estimate $O(1/\sqrt {\log n})$ for the speed of convergence using the quasi-powers theorem of Hwang~\cite[Theorem IX.8]{flajolet_sedgewick_book} and even a complete Edgeworth expansion using~\cite[Theorem~3.2.2]{feray_modphi} or~\cite{kabluchko:2017}, but we do not need this in the following.

\subsection{Log-concavity of the \texorpdfstring{$r$}{r}-Lah distribution}
The main result of this section is the following
\begin{prop}\label{rLahislogconcave}
For any admissible triple of parameters $(n,k,r)$ the $r$-Lah distribution $\Lah(n,k)_r$ is log-concave, that is
\begin{align*}
\P[\Lah(n,k)_r =i]^2 \geq \P[\Lah(n,k)_r =i-1]\P[\Lah(n,k)_r =i+1]
\end{align*}
for all $i  \in \lbrace k+1, \dotsc ,n-1\rbrace$.
\end{prop}

To prove this proposition, we need to generalize the log-concavity of $r$-Stirling numbers that is known~\cite{mezo:2007} for $r \in \N$ to arbitrary $r \in [0,\infty)$.
A sequence $(a_k)_{k=0}^n$ is called strictly log-concave if $a_k^2 > a_{k+1}a_{k-1}$ for all $k \in \{1,\ldots, n-1\}$.

\begin{lem} \label{logconcaveStirlingfirst}
For $r \in [0,\infty)$ and $n\in \N$, the sequence
\begin{align*}
\left( \stirfir{n}{k}_{r}\right)_{k=0}^n
\end{align*}
is strictly log-concave.
\end{lem}

\begin{proof}
For $r\in \N_0$, the claim has been proven in~\cite[Theorem~2]{mezo:2007}. The proof is based on the formula
\begin{align*}
A_{n,r}(x):= \sum_{j=0}^n \stirfir{n}{j}_r x^j  = (x+r)(x+r+1)\dotsm(x+r+n-1)
\end{align*}
which can be found in~\cite[Corollary~9]{broder:1984} for $r\in  \N_0$ but stays true for all $r\geq 0$ since both sides are polynomials in $r$.
Thus, all roots of $A_{n,r}(x)$ are real and the claim follows from Newton's inequality; see, e.g.\ \cite[Theorem~1]{mezo:2007}.
\end{proof}

\begin{lem} \label{logconcaveStirlingsecond}
For $r \in [0,\infty)$ and $n\in \N$,  the sequence
\begin{align*}
\left( \stirsec{n}{k}_{r} \right)_{k=0}^n
\end{align*}
is strictly log-concave.
\end{lem}
\begin{proof}
For $r\in \N_0$, the statement has been proven by Mez\H{o}~\cite[Theorem 6]{mezo:2007}. To adapt his proof to arbitrary $r > 0$,  we consider the polynomial
\begin{align*}
B_{n,r}(x)  \coloneqq \sum_{j=0}^n \stirsec{n}{j}_r x^j
\end{align*}
and show that its roots are all real (and negative). Using the recurrence relation~\eqref{recurrence_stirling2} we can write
\begin{alignat}{2}
& B_{n,r}(x) &=&  \sum_{j=0}^n \left[ (j+r) \stirsec{n-1}{j}_{r} + \stirsec{n-1}{j-1}_{r} \right] x^j \nonumber \\
&            &=& x \left( \frac{d}{dx} B_{n-1,r}(x) + B_{n-1,r}(x) \right) + r B_{n-1,r}(x). \label{BnrxDiff}
\end{alignat}
If $x\geq 0$, then $B_{n,r}(x) \geq B_{n,r}(0) =  \stirsec{n}{0}_r = r^n >0$  for all $n\in \N$. Thus, let $x < 0$ from now on.
Then (\ref{BnrxDiff}) can be rewritten as
\begin{align}\label{Bnr(x)form}
B_{n,r}(x) = \frac{x}{\e^x(-x)^r} \frac{\diff}{dx}  \left(\e^x (-x)^r B_{n-1,r}(x) \right).
\end{align}
We continue with induction over $n$. The only root of $B_{1,r}(x) = x+r$ is real and negative since $r>0$. Assume that $B_{n-1,r}(x)$ has $n-1$  negative roots. Then, $\e^x (-x)^r B_{n-1,r}(x)$ has $n-1$ negative roots and a root at $x=0$. By Rolle's theorem, $\frac{d}{dx} \Big(\e^x (-x)^r B_{n-1,r}(x)\Big)$ has at least $n-1$ negative roots. Equation \eqref{Bnr(x)form}  implies that $B_{n,r}(x)$ has at least $n-1$ negative roots. The remaining $n$-th root of this polynomial must be real (and negative), since $B_{n,r}(x)$ has real (and positive) coefficients and its roots come in complex conjugate pairs.   This completes the induction. Newton's inequality then yields the claim.
\end{proof}
\begin{lem} \label{lemstirsecforlogconcaverlah}
For $k\in \N_0$, real $r\geq 0$ and integer $j > k$ it holds that
$$
\stirsec{j}{k}_{r}^2 \geq \stirsec{j-1}{k}_{r}\stirsec{j+1}{k}_{r}.
$$
Equality only holds for $k=0$, while strict inequality holds for all $k\geq 1$.
\end{lem}
\begin{proof}
The cases $k=0$ and $k=1$ follow from the equations $\stirsec{j}{0}_{r} = r^j$ and $\stirsec{j}{1}_{r} = (r+1)^j-r^j$ which can be found in~\cite{nyul:2015}. For $k \geq 2$ we use~\eqref{recurrence_stirling2}:
\begin{alignat}{1}
&\stirsec{j}{k}_{r}\stirsec{j}{k}_{r} - \stirsec{j-1}{k}_{r}\stirsec{j+1}{k}_{r} \label{EqSeqStirSecdecres} \\
=& \stirsec{j}{k}_{r} \left[ (k+r)\stirsec{j-1}{k}_{r} + \stirsec{j-1}{k-1}_{r} \right] - \stirsec{j-1}{k}_{r} \left[ (k+r) \stirsec{j}{k}_{r} + \stirsec{j}{k-1}_{r} \right] \nonumber \\
=&\stirsec{j}{k}_{r}\stirsec{j-1}{k-1}_{r} - \stirsec{j-1}{k}_{r}\stirsec{j}{k-1}_{r}  \label{EqSeqStirsecincrea} \\
=&\left[ (k+r)\stirsec{j-1}{k}_{r} + \stirsec{j-1}{k-1}_{r} \right] \stirsec{j-1}{k-1}_{r} - \stirsec{j-1}{k}_{r} \left[ (k+r-1)\stirsec{j-1}{k-1}_{r} + \stirsec{j-1}{k-2}_{r} \right] \nonumber \\
=&\stirsec{j-1}{k-1}_{r}\stirsec{j-1}{k}_{r} + \stirsec{j-1}{k-1}_{r}^2  -\stirsec{j-1}{k}_{r}\stirsec{j-1}{k-2}_{r} > 0. \nonumber
\end{alignat}
The last statement is true because of Lemma \ref{logconcaveStirlingsecond}.
\end{proof}

Equations~\eqref{EqSeqStirSecdecres} and~\eqref{EqSeqStirsecincrea} yield the following corollary which, for $r=0$, can be found in Theorem~3.3 and Corollary~3.1 of~\cite{sibuya}.
\begin{kor}
For all $k \in \N$ and $r \in [0,\infty)$
the sequence $({\stirsec{j+1}{k}_{r}}/{\stirsec{j}{k}_{r}})_{j \geq k}$ is strictly decreasing, while the sequence $(\stirsec{j}{k+1}_{r}/{\stirsec{j}{k}_{r}})_{j\geq k}$ is strictly increasing.
\end{kor}

\begin{proof}[Proof of Proposition \ref{rLahislogconcave}]
The claim follows by multiplying the  inequalities of Lemmas~\ref{logconcaveStirlingfirst} and~\ref{lemstirsecforlogconcaverlah}.
\end{proof}

The log-concavity of the $r$-Lah distribution yields the following

\begin{kor}[Unimodality]
For any admissible triple of parameters $(n,k,r)$ the $r$-Lah distribution is unimodal. That is, there exists $m_{n,k,r} \in \lbrace k,\dotsc,n \rbrace$ such that $i \mapsto \P[\Lah(n,k)_r=i]$ is nondecreasing for $i \leq m_{n,k,r}$ and nonincreasing for $i \geq m_{n,k,r}$.
\end{kor}

\subsection{Further limit theorems}
Fix $k \in \N_0$ and $r\geq 0$ such that $\max\{k,r\} > 0$.
To derive additional limit theorems for the $r$-Lah distribution we will verify conditions A1-A4 of~\cite{kabluchko:2017} with the parameters
\begin{alignat}{4}
& \L_n(j) = \P [\Lah(n,k)_r = j], & w_n = (k+r)\log n, \nonumber \\
& \beta_{\pm} = \pm \infty,      & \mathcal{D} = \{\beta \in \C: |\Im \beta|<\pi/2\} , \nonumber \\
&\varphi( \beta) = \e^\beta -1,    & W_{\infty}(\beta) = \Gamma(k+2r)/\Gamma((k+r)\e^\beta + r) \nonumber.
\end{alignat}
Assumption A1 is trivial, while A2 and A3  follow from Theorem~\ref{rLahismodPoiss}. It remains to check A4.  We recall from Lemma \ref{DarstellungPnkr} that
\begin{align*}
P_{n,k,r} \Bigl( \e^{\beta + iu} \Bigr)
&=
\E \biggl[ \e^{(\beta + iu)\Lah(n,k)_r} \biggr]\\
&= \frac{1}{\binom{n+2r-1}{k+2r-1}} \sum_{m=0}^k (-1)^{k-m} \binom{k}{m} \frac{\Gamma\Bigl(r\bigl(\e^{\beta + iu}+1\bigr) + \e^{\beta + iu}m + n\Bigr)}{\Gamma\Bigl(r\bigl(\e^{\beta + iu}+1\bigr) + \e^{\beta + iu}m\Bigr)n!}.
\end{align*}
Let $K \subset \R$ be an arbitrary compact set and $a\in (0,\pi)$. Uniformly over $\beta \in K$, we have
\begin{alignat*}{2}
     &n^{-(k+r)(\e^\beta - 1)} \int_a^\pi \biggl| P_{n,k,r} \Bigl( \e^{\beta + iu} \Bigr) \biggr| \diff u \\
\leq &C n^{-(k+r)(\e^\beta - 1)} \sum_{m=0}^k  \int_a^\pi \biggl| \frac{\Gamma\Bigl(r\bigl(\e^{\beta + iu}+1\bigr) + \e^{\beta + iu}m + n\Bigr)}{n^{k+2r-1}n!} \biggr| \diff u \\
\leq &C\sum_{m=0}^k  \int_a^\pi \biggl| n^{\e^{\beta + iu}(r+m) - (k+r)\e^{\beta}} \biggr| \diff u \leq C n^{-\delta}
\end{alignat*}
for some constants $C=C(K,k) > 0$ and $\delta = \delta(K,k,a)>0$.   We used that $1/\Gamma(z)$ is bounded on every compact subset of $\C$ and, in the last line, \eqref{Gammaalphabeta}.

Now we can use the results from \cite{kabluchko:2017} to prove  limit theorems for the $r$-Lah distribution.
\begin{satz}[Local limit theorem]
For every fixed $k \in \N_0$ and $r \in [0,\infty)$ such that $\max\{k,r\} > 0$ we have
\begin{align*}
\sqrt{\log n} \sup_{j \in \Z} \Biggl| \P \bigl[\Lah(n,k)_r = j\bigr] - \frac{\exp \Bigl\{- \frac{( j - (k+r)\log n )^2}{2(k+r)\log n} \Bigr\}}{\sqrt{2 \pi (k+r) \log n}}  \Biggr| \ton 0.
\end{align*}
\end{satz}

\begin{proof}
The statement follows directly from Theorem 2.7 in \cite{kabluchko:2017}.
\end{proof}

\begin{prop}[Location of the mode]\label{location_mode}
For all $r \in [0,\infty)$ and every fixed $k \in \N_0$ such that $\max\{k,r\} > 0$ there is $N \in \N$ such that for all integers $n > N$, all maximizers of the function $i \mapsto \P[\Lah(n,k)_r = i]$ are among the following two numbers:
\begin{align*}
\biggl\lfloor (k+r)\log n - \frac{(k+r)\Gamma'(k+2r)}{\Gamma(k+2r)} -\frac{1}{2} \biggr\rfloor , \biggl\lceil (k+r)\log n - \frac{(k+r)\Gamma'(k+2r)}{\Gamma(k+2r)} -\frac{1}{2} \biggr\rceil.
\end{align*}
\end{prop}

\begin{proof}
The claim follows from Theorem 2.11 and Remark 2.12 in \cite{kabluchko:2017}.
\end{proof}

\begin{satz}[Precise asymptotics of large deviations]\label{theo:LDP}
Let $k \in \N_0$ and $r \in [0,\infty)$ be such that $\max\{k,r\} > 0$ and let $(x_n)_{n \in \N}$ be a sequence of real numbers converging to $x > 0$, such that $(k+r)x_n \log n$ is integer for all $n \in \N$. Then, we have the asymptotics
\begin{align*}
\P \bigl[ \Lah(n,k)_r    = (k+r)x_n\log n \bigr]
&\sim
\frac{n^{-(k+r)\bigl(x_n \log(x_n) -x_n +1\bigr)}}{\sqrt{2 \pi (k+r) x \log n}} \frac{\Gamma(k+2r)}{\Gamma\bigl((k+r)x+r\bigr)}, \\
\P\bigl[\Lah(n,k)_r \geq (k+r)x_n\log n\bigr]
&\sim
\frac{x}{x-1} \frac{n^{-(k+r)\bigl(x_n \log(x_n) -x_n +1\bigr)}}{\sqrt{2 \pi (k+r) x \log n}} \frac{\Gamma(k+2r)}{\Gamma\bigl((k+r)x+r\bigr)}, \\
\P\bigl[\Lah(n,k)_r \leq (k+r)x_n\log n\bigr]
&\sim
\frac{1}{1-x} \frac{n^{-(k+r)\bigl(x_n \log(x_n) -x_n +1\bigr)}}{\sqrt{2 \pi (k+r) x \log n}} \frac{\Gamma(k+2r)}{\Gamma\bigl((k+r)x+r\bigr)}.
\end{align*}
The second statement only holds for $x >1$, while the third only holds for $x<1$.
\end{satz}

\begin{proof}
The first statement follows from Remark 2.9 in \cite{kabluchko:2017}, with $K \subset \R$ being an arbitrary compact segment which contains $x$ and $\log(x_n)$. For the other two asymptotics, c.f.\ Theorem~3.2.2 in~\cite{feray_modphi} (or repeat the proof of Theorem~2.8 in \cite{kabluchko:2017} starting with the second formula of Lemma~3.2.1 of~\cite{feray_modphi}).
\end{proof}

\section{A threshold phenomenon for face numbers of cones generated by random walks}\label{threshold_phenomena}
We will use the results from the previous section to prove weak and strong threshold phenomena for the expected number of $k$-faces of the cone $C_n^B$ that has been defined in the introduction. As already stated, the following equation holds for all $k\in \{0,\ldots, d-1\}$ and $n\geq d$:
\begin{align*}
\E \Bigl[ f_k \bigl( C_n^B \bigr) \Bigr] = \frac{2\cdot k!}{n!} \sum_{l=0}^\infty \stirfir{n}{d-2l-1}_{\frac{1}{2}}\stirsec{d-2l-1}{k}_{\frac{1}{2}},
\end{align*}
see \cite[Corollary~2.12]{godland:2022} and~\cite[Equation~(2.23)]{godlandschlaefli:2022}.
We can represent the right-hand side in terms of the $r$-Lah distribution with $r=1/2$
\begin{alignat}{2}
&\frac{\E\bigl[ f_{k}(C_{n}^B) \bigr]}{\binom{n}{k}} &=&2 \frac{k!}{\binom{n}{k}n!} \sum_{l=0}^\infty \stirfir{n}{d-2l-1}_{\frac{1}{2}}\stirsec{d-2l-1}{k}_{\frac{1}{2}}\nonumber \\
&&=&2 \P\left[\Lah(n,k)_{\frac{1}{2}} \in \lbrace d-1,d-3, \dotsc \rbrace\right]. \label{Equation_d-1}
\end{alignat}
Since it holds that
\begin{itemize}
\item $\displaystyle \sum_{j=k}^{n} (-1)^{n-j}\stirfir{n}{j}_{r}\stirsec{j}{k}_{r} = 0$ if $n>k$ (c.f.\ Theorem 5 in \cite{broder:1984}),
\item $\displaystyle\sum_{j=k}^{n} \stirfir{n}{j}_{r}\stirsec{j}{k}_{r} = L(n,k)_r$,
\end{itemize}
we conclude that,  for $n>k$,
\begin{align}\label{evenodd}
\P[\Lah(n,k)_r \text{ takes even value}] = \P[\Lah(n,k)_r \text{ takes odd value}] = \frac{1}{2}.
\end{align}
This allows us to rewrite (\ref{Equation_d-1}) as
\begin{align}  \label{Complement_Equation_d-1}
1 - \frac{\E\bigl[ f_{k}(C_{n}^B)\bigr]}{\binom{n}{k}} = 2 \P\left[\Lah(n,k)_{\frac{1}{2}} \in \lbrace d+1,d+3, \dotsc \rbrace\right].
\end{align}
These representations will be used to prove threshold phenomena; c.f.\ \cite[Theorems~6.2, 6.3]{kabluchko:2022} for corresponding  results in the case $r=0$.

\begin{satz}[Weak Threshold]\label{Treshold}
Let $d \to \infty$ and $n=n(d)\geq d$ be a function of $d$ such that
\begin{align*}
\gamma \coloneqq \lim_{d \to \infty} \frac{\log\bigl(n(d)\bigr)}{d} \in [0,\infty].
\end{align*}
Then, for every fixed $k \in \N_0$, we have
\begin{align*}
\lim_{d \to \infty} \frac{\E\bigl[ f_{k}(C_{n}^B) \bigr]}{\binom{n}{k}} = \begin{cases}1, &\text{if }\quad \gamma < \frac{1}{k+\frac{1}{2}},\\
0, &\text{if } \quad \gamma > \frac{1}{k+\frac{1}{2}}.
\end{cases}
\end{align*}
In the critical case  $\gamma = \frac{1}{k+\frac{1}{2}}$, if
\begin{align}\label{eq:log_n_critical_case}
\log\bigl( n(d) \bigr) = \biggl(\frac{1}{k+\frac{1}{2}}\biggr)\biggl( d+ c \sqrt{d} + o\Bigl( \sqrt{d} \Bigr) \biggr)
\end{align}
for  some constant $c\in\R$, then
\begin{align*}
\lim_{d \to \infty} \frac{\E\bigl[ f_{k}(C_{n}^B) \bigr]}{\binom{n}{k}} = \frac{1}{\sqrt{2\pi}} \int_{-\infty}^{-c} \e^{-x^2/2} \diff x.
\end{align*}
\end{satz}

\begin{proof}
First, let $\gamma > \frac{1}{k+ \frac{1}{2}}$. Then,
\begin{alignat*}{2}
&\frac{\E\bigl[ f_{k}(C_{n}^B) \bigr]}{\binom{n}{k}} &=&2 \P\left[\Lah(n,k)_{\frac{1}{2}} \in \lbrace d-1,d-3, \dotsc \rbrace\right] \\\
&&\leq & 2 \P\left[\Lah(n,k)_{\frac{1}{2}} \leq d\right] \nonumber \\
&&=& 2 \P \left[ \frac{\Lah(n,k)_{\frac{1}{2}} - \left(k+\frac{1}{2}\right)\log n}{\sqrt{\left(k+\frac{1}{2}\right)\log n}} \leq  \frac{d - \left(k+\frac{1}{2}\right)\log n}{\sqrt{\left(k+\frac{1}{2}\right)\log n}}\right] \nonumber \\
&&=&2 \P \left[ \frac{\Lah(n,k)_{\frac{1}{2}} - \left(k+\frac{1}{2}\right)\log n}{\sqrt{\left(k+\frac{1}{2}\right)\log n}} \leq \frac{\sqrt{\log n} \left( \frac{1}{\gamma} + o(1) -\left(k+\frac{1}{2}\right) \right)}{\sqrt{k+\frac{1}{2}}} \right]. \nonumber
\end{alignat*}
Since  $\frac{1}{\gamma} < k + \frac{1}{2}$, the last fraction on the right-hand side goes to $-\infty$ and the probability  goes to $0$ by the CLT for the $r$-Lah distribution, see Theorem~\ref{thm:clt_lah_r}.

Let now  $\gamma < \frac{1}{k+\frac{1}{2}}$.  Equation (\ref{Complement_Equation_d-1}) gives
\begin{alignat*}{2}
&1-\frac{\E\left[ f_{k}(C_{n}^B) \right]}{\binom{n}{k}}& \leq& 2\P\left[ \Lah(n,k)_{\frac{1}{2}} \geq d \right] \\
&&=&2 \P \left[ \frac{\Lah(n,k)_{\frac{1}{2}} - \left(k+\frac{1}{2}\right)\log n}{\sqrt{\left(k+\frac{1}{2}\right)\log n}} \geq \frac{\sqrt{\log n} \left( \frac{1}{\gamma}+o(1) -\left(k+\frac{1}{2}\right) \right) }{\sqrt{k+\frac{1}{2}}} \right].
\end{alignat*}
Since $\gamma < \frac{1}{k+\frac{1}{2}}$, the last fraction on the right-hand side goes to $+\infty$ and Theorem~\ref{thm:clt_lah_r} implies that the probability goes to $0$.

Finally, we consider the critical case when~\eqref{eq:log_n_critical_case} holds. If the right-hand side of (\ref{Equation_d-1}) could be replaced with $\P[\Lah(n,k)_{\frac{1}{2}} \leq d]$, we could once again use Theorem~\ref{thm:clt_lah_r} in the following way:
\begin{align*}
\P[\Lah(n,k)_{\frac{1}{2}} \leq d]
&=
\P \left[ \frac{\Lah(n,k)_{\frac{1}{2}} - \left(k+\frac{1}{2}\right)\log n}{\sqrt{\left(k+\frac{1}{2}\right)\log n}} \leq  \frac{d - \left(k+\frac{1}{2}\right)\log n}{\sqrt{\left(k+\frac{1}{2}\right)\log n}}\right]
\\
&=
\P \left[ \frac{\Lah(n,k)_{\frac{1}{2}} - \left(k+\frac{1}{2}\right)\log n}{\sqrt{\left(k+\frac{1}{2}\right)\log n}} \leq -c+o(1) \right]
\\
&\underset{{d \to \infty}}{\longrightarrow} \frac{1}{\sqrt{2\pi}} \int_{- \infty}^{-c} \e^{-x^2/2} \diff x.
\end{align*}
The justification of this replacement can be done as follows (c.f.\ \cite[Theorem~6.2]{kabluchko:2022}).
By Proposition \ref{location_mode} the largest mode of the $r$-Lah distribution $m_{n,k,r}$ for $r=1/2$ satisfies
\begin{align*}
m_{n,k,\frac{1}{2}} = \left( k+\frac{1}{2} \right) \log n + O(1) = d + c\sqrt{d} + o(\sqrt{d}).
\end{align*}
Thus, if $c>0$,  it follows that $d<m_{n,k,\frac{1}{2}}$ for sufficiently large $d$. 	As the $r$-Lah distribution is unimodal, the function $j \mapsto \P[\Lah(n,k)_{\frac{1}{2}} = j]$ is nondecreasing for all $j \leq d$. We conclude that
\begin{alignat*}{2}
&\P[\Lah(n,k)_{\frac{1}{2}} \leq d-2] & \leq & 2\P[\Lah(n,k)_{\frac{1}{2}} \in \lbrace d-1, d-3,\dotsc \rbrace] \\
&&\leq & \P[\Lah(n,k)_{\frac{1}{2}} \leq d].
\end{alignat*}
Applying Theorem~\ref{thm:clt_lah_r} to both sides of the inequality completes the proof for $c >0$.
The case  $c<0$ is similar using~\eqref{Complement_Equation_d-1}, while the case $c=0$ follows by a sandwich argument.
\end{proof}

\begin{bem}
Take  $k=0$ in Theorem~\ref{Treshold}. Then $f_{0}(C_{n}^B) = 0$ or $1$ depending on whether $C_{n}^B=\R^d$ or not. In this case, Theorem~\ref{Treshold} recovers Corollary~5.4 from~\cite{kabluchko:2016}.
\end{bem}

\begin{satz}[Strong Threshold]\label{strong_treshold}
For fixed $k \in \N_0$, let $n=n(d) \geq  d$ be an integer sequence. If $n(d) \leq \e^{d/((k+ \frac{1}{2})e)}$ for all sufficiently large $d$, then
the random cone $C_n^B$ is $k$-neighbourly with probability converging to  $1$, more precisely
\begin{align}
\P \left[ f_{k}(C_{n}^B) = \binom{n}{k} \right] = 1 - O\left(n^{-\frac 12}\right).
\end{align}
\end{satz}

\begin{proof}
On the event $f_k(C_n^B) \neq \binom{n}{k}$ it holds that $\binom{n}{k} - f_k(C_n^B) \geq 1$, hence
\begin{align*}
\P \left[ f_k(C_n^B) \neq \binom{n}{k} \right] \leq \E \left[ \binom{n}{k} - f_k(C_n^B) \right] =
\binom{n}{k} \left( 1- \frac{\E[f_k(C_n^B)]}{\binom{n}{k}} \right).
\end{align*}
In view of~\eqref{Complement_Equation_d-1} this gives us
\begin{align*}
\P \left[ f_k(C_n^B) \neq \binom{n}{k} \right]  \leq 2n^k \P \left[ \Lah(n,k)_{\frac{1}{2}} \geq (k+ 1/2) x_n \log n \right]
\end{align*}
with $x_n \coloneqq d/((k+1/2)\log n)$. Note that  $x_n \geq e$ by assumption. Applying Theorem~\ref{theo:LDP} we get
\begin{align*}
\P \left[ f_k(C_n^B) \neq \binom{n}{k} \right] \leq 2 n^k n^{-(k+\frac{1}{2})(x_n\log(x_n)-x_n+1)}
\end{align*}
(the missing fractions can be bounded above by $1$).
Since $f(x)\coloneq x\log x-x+1$ is strictly increasing for $x>1$, we have $f(x_n)  \geq  f(e) = 1$. Thus
\begin{align*}
\P \left[ f_k(C_n^B) \neq \binom{n}{k} \right] \leq 2 n^{k}n^{-(k+\frac 12)} = O\left(n^{-\frac 12}\right),
\end{align*}
which proves the statement.
\end{proof}

\section{Application to compressed sensing of monotone signals}\label{sec:compressed_sensing}
Now we present an application of Theorem \ref{Treshold} to compressed sensing.  The following general setting has been introduced by Donoho and Tanner~\cite{donoho:2010}. Let $P\subseteq \R^n$ be a polyhedral set, that is, an intersection of finitely many closed half-spaces (whose bounding hyperplanes need not pass through the origin). For an unknown signal $x= (x_1,\dotsc ,x_n)$ from  $P$ and a Gaussian random matrix $G : \R^n \to \R^d$ with $d \leq n$ we are interested in uniquely recovering the signal from its image $y=Gx$.
This event is denoted by $\text{Unique}(G,x,P)$ and takes place if and only if $(x+ \ker G ) \cap P = \{x\}$.
Donoho and Tanner computed the probability of unique recovery for $P=\R_+^n$ (and also $P=[0,1]^n$) if $x$ is $k$-sparse (meaning that it has $k$ non-zero entries) and uncovered threshold phenomena for this probability in the linear growth regime of $n,k,d$. In the setting when $k\in \N_0$ is constant and $d,n\to\infty$ such that $d/n \to \delta \in [0,1]$ it is known~\cite[Proposition~4.3]{godland:2020}
that
\begin{equation}\label{eq:unique_probab_donoho_tanner}
\lim_{n\to\infty} \P[\text{Unique}(G,x,\R_+^n)] =
\begin{cases}
1, &\text{if }\quad \delta > 1/2 ,\\
0, &\text{if } \quad \delta < 1/2.
\end{cases}
\end{equation}
Using the above results on the $r$-Lah distribution with $r=1/2$ we will study a similar problem for \textit{monotone} sparse signals. To be precise, our aim is to recover a monotone signal $x = (x_1, \dotsc ,x_n)$  that belongs to the Weyl chamber $B^{(n)} \coloneqq \lbrace x \in \R^n | x_1 \geq x_2 \geq \dotsc \geq x_n \geq 0 \rbrace $. We use the following probabilistic model for $x$. Let $0 < k \leq n$ be given together with positive numbers $a_1, \dotsc ,a_k$. Select a random uniform subset $\lbrace i_1,\dotsc, i_k \rbrace$ of $\lbrace 1,\dotsc ,n \rbrace$ with $1 \leq i_1 < \dotsc < i_k \leq n$
and define the signal $x=(x_1,\dotsc,x_n)$ by
$$
x_m \coloneqq \sum_{l=1}^k a_l \ind_{i_l\geq m} \qquad \text{for }m=1,\dotsc,n.
$$
Note that the entries of $x$ are monotone non-increasing (and non-negative) and the sparsity parameter $k$ gives the number of jumps in  $x$. For $k=0$ we just put $x=0$.
Proposition~3.17 of~\cite{godlandschlaefli:2022} calculates the probability of the unique recovery event in the above setting to be
$$
\P[\text{Unique}(G,x,B^{(n)})] = 2 \frac{k!}{ n! \binom{n}{k}} \sum_{l=0}
^\infty \stirfir{n}{d-2l-1}_{\frac{1}{2}} \stirsec{d-2l-1}{k}_{\frac{1}{2}},
$$
for $0\leq k \leq d \leq n$, and thus  it follows from~\eqref{Equation_d-1} that
$$
\P[\text{Unique}(G,x,B^{(n)})] = \frac{\E\left[ f_k(C_n^B) \right]}{\binom{n}{k}}.
$$
Therefore, Theorem \ref{Treshold}  yields the following threshold phenomenon for the event of unique recovery.
\begin{satz}\label{Treshold_sensing}
Fix $k\in \N_0$. Let $d \to \infty$ and $n=n(d)\geq d$ be a function of $d$ with
\begin{align*}
\gamma \coloneqq \lim_{d \to \infty} \frac{\log\bigl(n(d)\bigr)}{d} \in [0,\infty].
\end{align*}
Let $x \in \R^n$ be a random signal in the Weyl chamber $B^{(n)}$ constructed as above. Finally, let $G : \R^n \to \R^d$ be a Gaussian random matrix.
Then,
\begin{align*}
\lim_{d \to \infty} \P[\mathrm{Unique}(G,x,B^{(n)})] = \begin{cases}1, &\text{if }\quad \gamma < \frac{1}{k+\frac{1}{2}},\\
0, &\text{if } \quad \gamma > \frac{1}{k+\frac{1}{2}}.
\end{cases}
\end{align*}
In the critical case $\gamma = \frac{1}{k+\frac{1}{2}}$, if
\begin{align*}
\log\bigl( n(d) \bigr) = \biggl(\frac{1}{k+\frac{1}{2}}\biggr)\biggl( d+ c \sqrt{d} + o\Bigl( \sqrt{d} \Bigr) \biggr)
\end{align*}
for some constant $c\in \R$, then
\begin{align*}
\lim_{d \to \infty} \P[\operatorname{Unique}(G,x,B^{(n)})] = \frac{1}{\sqrt{2\pi}} \int_{-\infty}^{-c} \e^{-y^2/2} \diff y.
\end{align*}
\end{satz}

\begin{bem}
For monotone signals in the  type $A$ Weyl chamber
$$
A^{(n)} \coloneqq \lbrace x \in \R^n | x_1 \geq x_2 \geq \dotsc \geq x_n\rbrace
$$
 a similar analysis is possible using Proposition~3.18 in~\cite{godlandschlaefli:2022} and replacing $r=1/2$ by $r=0$ throughout. For two more examples with $r=0$, see Propositions~3.19, 3.20  in~\cite{godlandschlaefli:2022}.
\end{bem}
Let us now compare this statement on unique recovery of $k$-sparse monotone signals to the result~\eqref{eq:unique_probab_donoho_tanner} on unique recovery of $k$-sparse non-negative signals, with fixed $k$.
Roughly speaking, for non-negative  signals, \eqref{eq:unique_probab_donoho_tanner} states that unique recovery  is possible if $d$ is at least $n/2$. On the other hand, for monotone signals, Theorem~\ref{Treshold_sensing} states that unique recovery  is possible if $d$ is larger than $(k+\frac 12)\log n$. Clearly, one can pass from a non-negative signal to a monotone one by building the partial sums of its entries, which does not change the sparsity parameter $k$.  The above results suggest that these partial sums can be compressed at a much higher rate than the original signal.

\subsection*{Acknowledgement}
We are grateful to Hsien-Kuei Hwang for pointing out reference~\cite{van_der_hofstad_etal_shortest_path_trees}.
Supported by the German Research Foundation under Germany’s Excellence Strategy EXC 2044 – 390685587, Mathematics M\"unster: Dynamics - Geometry - Structure and by the DFG priority program SPP 2265 Random Geometric Systems.

\bibliographystyle{plain} 
\bibliography{refs} 

\begin{thebibliography}{10}

\bibitem{barbour:2014}
A.~D. Barbour, E.~Kowalski, and A.~Nikeghbali.
\newblock Mod-discrete expansions.
\newblock {\em Probability Theory and related fields}, 158(3-4):859--893, 2014.

\bibitem{belbachir:2013}
H.~Belbachir and A.~Belkhir.
\newblock Cross recurrence relations for r-{L}ah numbers.
\newblock {\em Ars Combinatoria}, 110:199--203, 2013.

\bibitem{broder:1984}
A.~Broder.
\newblock The $r$-{S}tirling numbers.
\newblock {\em Discrete Mathematics}, 49(3):241--259, 1984.

\bibitem{charalambides_book_enum_combin}
C.~A. Charalambides.
\newblock {\em Enumerative combinatorics}.
\newblock Chapman \& Hall/CRC, 2002.

\bibitem{charalambides_book_combinatorial_methods}
C.~A. Charalambides.
\newblock {\em Combinatorial methods in discrete distributions}.
\newblock Wiley Ser. Probab. Stat. John Wiley \& Sons, 2005.

\bibitem{cheon:2012}
G.-S. Cheon and J.-H. Jung.
\newblock $r$-{W}hitney numbers of {D}owling lattices.
\newblock {\em Discrete Mathematics}, 312:2337--2348, 2012.

\bibitem{delbaen:2015}
F.~Delbaen, E.~Kowalski, and A.~Nikeghbali.
\newblock Mod-$\phi$ convergence.
\newblock {\em International Mathematics Research Notices},
  2015(11):3445--3485, 2015.

\bibitem{donoho_neighborliness_proportional}
D.~L. Donoho.
\newblock High-dimensional centrally symmetric polytopes with neighborliness
  proportional to dimension.
\newblock {\em Discrete and Computational Geometry}, 35(4):617--652, 2006.

\bibitem{donoho_tanner_neighborliness}
D.~L. Donoho and J.~Tanner.
\newblock Neighborliness of randomly projected simplices in high dimensions.
\newblock {\em Proc. Natl. Acad. Sci. USA}, 102(27):9452--9457, 2005.

\bibitem{donoho_tanner}
D.~L. Donoho and J.~Tanner.
\newblock Counting faces of randomly projected polytopes when the projection
  radically lowers dimension.
\newblock {\em Journal of the American Mathematical Society}, 22(1):1--53,
  2009.

\bibitem{donoho:2010}
D.~L. Donoho and J.~Tanner.
\newblock Counting the faces of randomly-projected hypercubes and orthants,
  with applications.
\newblock {\em Discrete Computational Geometry}, 43(3):522--541, 2010.

\bibitem{feray_modphi}
V.~F\'eray, P.~M\'eliot, and A.~Nikeghbali.
\newblock {\em Mod-$\phi$ convergence}.
\newblock SpringerBriefs in Probability and Mathematical Statistics, 1st
  edition, 2016.

\bibitem{fields:1970}
J.~L. Fields.
\newblock The uniform asymptotic expansion of a ratio of two gamma functions.
\newblock In {\em Proceedings of the International Conference on Constructive
  Function Theory}, pages 171--176, Varna, Bulgaria, 1970.

\bibitem{flajolet_sedgewick_book}
P.~Flajolet and R.~Sedgewick.
\newblock {\em Analytic combinatorics}.
\newblock Cambridge University Press, Cambridge, 2009.

\bibitem{godlandschlaefli:2022}
T.~Godland and Z.~Kabluchko.
\newblock Angle sums of {S}chläfli orthoschemes.
\newblock {\em Discrete and Computational Geometry}, 68:125--164, 2022.

\bibitem{godland:2022}
T.~Godland and Z.~Kabluchko.
\newblock Positive hulls of random walks and bridges.
\newblock {\em Stochastic Processes and their Applications}, 147:327--362,
  2022.

\bibitem{godland:2020}
T.~Godland, Z.~Kabluchko, and C.~Th\"ale.
\newblock Random cones in high dimensions {I}: {D}onoho-{T}anner and
  {C}over-{E}fron cones.
\newblock {\em Discrete Analysis}, 2020(5), 2020.

\bibitem{graham_etal_book}
R.~L. Graham, D.~E. Knuth, and O.~Patashnik.
\newblock {\em Concrete mathematics}.
\newblock Addison-Wesley Publishing Company, second edition, 1994.

\bibitem{jacod:2011}
J.~Jacod, E.~Kowalski, and A.~Nikeghbali.
\newblock Mod-{G}aussian convergence: new limit theorems in probability and
  number theory.
\newblock {\em Forum of Mathematics}, 23(4):835--873, 2011.

\bibitem{kabluchko:2022}
Z.~Kabluchko and A.~Marynych.
\newblock Lah distribution: Stirling numbers, records on compositions, and
  convex hulls of high-dimensional random walks.
\newblock {\em Probability Theory and Related Fields}, 184:969--1028, 2022.

\bibitem{kabluchko:2017}
Z.~Kabluchko, A.~Marynych, and H.~Sulzbach.
\newblock General {E}dgeworth expansions with applications to profiles of
  random trees.
\newblock {\em Annals of Applied Probability}, 27(6):3478--3524, 2017.

\bibitem{kabluchko:2016}
Z.~Kabluchko, V.~Vysotsky, and D.~Zaporozhets.
\newblock Convex hulls of random walks, hyperplane arrangements, and {W}eyl
  chambers.
\newblock {\em Geometric and Functional Analysis}, 27(4):880--918, 2017.

\bibitem{knuth_notation}
D.~E. Knuth.
\newblock Two notes on notation.
\newblock {\em Amer. Math. Monthly}, 99(5):403--422, 1992.

\bibitem{kowalski:2009}
E.~Kowalski and A.~Nikeghbali.
\newblock Mod-{G}aussian convergence and the value distribution of $\zeta (1/2
  + it)$ and related quantities.
\newblock {\em Journal of the London Mathematical Society}, 86(1):391--319, 12
  2009.

\bibitem{kowalski:2010}
E.~Kowalski and A.~Nikeghbali.
\newblock Mod-{P}oisson convergence in probability theory and number theory.
\newblock {\em International Mathematics Research Notices}, 18:3549--3587,
  2010.

\bibitem{lah:1954}
I.~Lah.
\newblock A new kind of numbers and its application in the actuarial
  mathematics.
\newblock {\em Boletim do Instituto dos Actu{\'a}rios Portugueses}, 9:7--15,
  1954.

\bibitem{merris:2000}
R.~Merris.
\newblock The $p$-{S}tirling numbers.
\newblock {\em Turkish Journal of Mathematics}, 24:379--399, 2000.

\bibitem{mezo:2007}
I.~Mez\H{o}.
\newblock On the maximum of $r$-{S}tirling numbers.
\newblock {\em Advances in Applied Mathematics}, 41(3):293--306, 2007.

\bibitem{mezo_book}
I.~Mez{\H{o}}.
\newblock {\em Combinatorics and number theory of counting sequences}.
\newblock CRC Press, 2020.

\bibitem{meliot:2014}
P.~Méliot and A.~Nikeghbali.
\newblock Mod-{G}aussian convergence and its applications for models of
  statistical mechanics.
\newblock In {\em In memoriam {M}arc {Y}or---{S}\'{e}minaire de
  {P}robabilit\'{e}s {XLVII}}, volume 2137 of {\em Lecture Notes in Math.},
  pages 369--425. Springer, Cham, 2015.

\bibitem{nyul:2015}
G.~Nyul and G.~R\'acz.
\newblock The $r$-{L}ah numbers.
\newblock {\em Discrete Mathematics}, 338(10):1660--1666, 2015.

\bibitem{petkovsek:2022}
M.~Petkovsek and T.~Pisanski.
\newblock Combinatorial interpretation of unsigned {S}tirling and {L}ah
  numbers.
\newblock {\em Pi Mu Epsilon Journal}, 12:417--424, 2007.

\bibitem{shattuck:2016}
M.~Shattuck.
\newblock Generalized {$r$}-{L}ah numbers.
\newblock {\em Proc. Indian Acad. Sci. Math. Sci.}, 126(4):461--478, 2016.

\bibitem{sibuya}
M.~Sibuya.
\newblock Log-concavity of {S}tirling numbers and unimodality of {S}tirling
  distributions.
\newblock {\em Ann. Inst. Statist. Math.}, 40(4):693--714, 1988.

\bibitem{van_der_hofstad_etal_shortest_path_trees}
R.~van~der Hofstad, G.~Hooghiemstra, and P.~Van~Mieghem.
\newblock Size and weight of shortest path trees with exponential link weights.
\newblock {\em Combin. Probab. Comput.}, 15(6):903--926, 2006.

\bibitem{vershik_sporyshev_asymptotic_faces_random_polyhedra1992}
A.~M. Vershik and P.~V. Sporyshev.
\newblock Asymptotic behavior of the number of faces of random polyhedra and
  the neighborliness problem.
\newblock {\em Selecta Math. Soviet.}, 11(2):181--201, 1992.

\end{thebibliography}

\end{document}